\documentclass[11pt]{article}
\usepackage{pgfplots}
\pgfplotsset{compat=1.15}
\usepackage{mathrsfs}
\usetikzlibrary{arrows}
\usepackage{amscd,amsmath,amsfonts,amssymb,amsthm,cite,mathrsfs,color} 
\usepackage{dsfont} \usepackage{leftidx}
\usepackage{tikz} \usepackage{enumerate}
\usepackage{appendix}

\usepackage{hyperref}

\usepackage[a4paper,left=25mm, right=25mm, top=25mm, bottom=25mm]{geometry}
\usepackage{graphicx} \usepackage{xcolor}

\numberwithin{equation}{section} 

\usepackage{amsthm}

\theoremstyle{plain}
\newtheorem{theorem}{Theorem}[section]
\newtheorem{lemma}[theorem]{Lemma}
\newtheorem{corollary}[theorem]{Corollary}

\theoremstyle{definition}

\newtheorem{remark}{Remark}[section]

\newtheorem*{acknow}{Acknowledgments}
\begin{document}
	
	\title{\bf{ Global and local limits for products of rectangular Ginibre matrices }}
	
	\author{
		Yandong Gu\footnotemark[1],   
	}
	\renewcommand{\thefootnote}{\fnsymbol{footnote}}
	\footnotetext[1]{School of Mathematical Sciences, University of Science and Technology of China, Hefei 230026, P.R.~China. E-mail: gd27@mail.ustc.edu.cn
		
	}

	\maketitle
	\begin{abstract}
We investigate singular value statistics for products of independent rectangular complex Ginibre matrices. When the rectangularity parameters of the matrices converge to a common limit in the asymptotic regime, the limiting spectral density is derived, and the local statistics  in  the bulk are shown to be governed by the universal sine kernel. This generalizes the classical results for products of square Ginibre matrices to a specific class of rectangular matrix products.
	\end{abstract}
	
	\section{Introduction and main results}
	
	\subsection{Introduction}
The study of products of random matrices dates back to the seminal work of Bellman \cite{Bel54} in 1954. Fundamental asymptotic results were later established by Furstenberg and Kesten \cite{FK60}, who developed laws of large numbers and central limit theorems for such products, extending classical probability theory to non-commutative settings.

In this work, we consider a product of independent rectangular complex Ginibre matrices:
\begin{equation}
	Y_M = X_M \cdots X_1,
\end{equation}
where each $X_j$ is of size $N_j \times N_{j-1}$. We associate the dimensions $N_0, N_1, \dots, N_M$ with a large integer parameter $N$ (also denoted as $N_0$), such that
\begin{align}
	\min\{N_0, \dots, N_M\} = N_0 = N,
\end{align}
and define the nonnegative integers
\begin{align}
	\nu_j = N_j - N_0, \quad j = 0, \dots, M.
\end{align}
Note that $\nu_0 = 0$ and $\nu_j \geq 0$ for $j = 1, \dots, M$. The exact joint density of the singular values of $Y_M$ was derived by Akemann, Kieburg, and Wei \cite{AKW13}, who showed that these singular values form a determinantal point process. Kuijlaars and Zhang \cite{KZ14} subsequently provided a double integral representation for the correlation kernel.

When the matrices are square, the statistical properties of the singular values are well-understood. The limiting spectral density follows the Fuss-Catalan distribution\cite{MS17}\cite{Ne14}, and the local correlations in the bulk are universal\cite{LWW23}\cite{LWZ16}, governed by the sine kernel-the same as for a single Ginibre matrix\cite{For2010}. This universality reflects the robustness of local spectral statistics. However, many applications involve inherently non-square transformations, where input and output dimensions differ. This motivates the study of products of rectangular Ginibre matrices. A key parameter is the rectangularity-the ratio of dimensions of each matrix. A central question is how rectangularity influences the spectral properties of the product. When the matrices have differing degrees of rectangularity, the analysis becomes particularly challenging, and the limiting behavior may be highly nontrivial.
In this work, we focus on the regime where the collective depth-to-width parameter 
\begin{equation}
	\Delta_{M,N} = \sum_{j=0}^{M} \frac{1}{N + \nu_j} \to 0
\end{equation}
introduced in \cite{GU25}\cite{LWW23}, indicating that the number of matrices is much smaller than their dimensions. We assume that
\begin{equation}
	\lim_{N\to \infty} \frac{N}{N_l} = y_l \in (0,1], \quad l=1,\dots,M,
\end{equation}
where the parameters $y_l$ capture the limiting rectangularity ratios and fundamentally shape the spectral distribution.
Our analysis proceeds in two main steps. First, we derive the limiting mean spectral density for the squared singular values of $Y_M$. The Stieltjes transform $G(z)$ of this limiting distribution satisfies the algebraic equation (see e.g.\cite{AGT10})
\begin{equation}\label{S_equ}
	1 - zG(z) + G(z) \prod_{l=1}^{M} \left(1 - y_l + z y_l G(z)\right) = 0.
\end{equation}
A careful study of this equation yields explicit expressions for the limiting density, providing the foundation for understanding the global distribution of singular values. Second, we use this density to analyze the local statistics in the bulk. The mean spectral density determines the appropriate scaling and is essential for the asymptotic analysis of the correlation kernel 
\begin{equation}\label{logkn}
	K_{M,N}(x,y) = \int_{c-i\infty}^{c+i\infty}\frac{ds}{2\pi i}\oint_{\Sigma}\frac{dt}{2\pi i}\frac{e^{xt-ys}}{s-t}\frac{\Gamma(t)}{\Gamma(s)}\prod_{j=0}^M\frac{\Gamma(s+N+\nu_j)}{\Gamma(t+N+\nu_j)}.
\end{equation} for the log-transformed matrix $\log(Y^*_M Y_M)$.
Here $\Sigma$ is a counter-clockwise contour encircling $0, -1, \dots, -N+1$ and $c$ is chosen
to make the vertical s-contour disjoint from $\Sigma$.

Our main result establishes that the local  statistics in the bulk  follow the sine kernel. This extends classical Wishart matrix results to non-square multiplicative chains, revealing how the rectangularity parameters $y_l$ govern both global and local spectral behavior.
A related limit phenomenon has been widely studied in diverse random matrix product ensembles \cite{AA22}\cite{AP23}\cite{BE25}\cite{GS22}\cite{GL25}
\cite{Jq17}\cite{LW24}, revealing universal  patterns in mathematical perspective of spectral statistics.

\subsection{Main results} 
We now consider the special case where $y_l = y$ for all $l$. Under this condition, equation \eqref{S_equ} simplifies to
\begin{equation}\label{algformula}
	1 - zG(z) + G(z) \left(1 - y + z y G(z)\right)^M = 0.
\end{equation}
Applying the change of variables $W = zG(z) + 1/y - 1$, we obtain
\begin{equation}\label{re_equa}
	y^M \left(W + 1 - \frac{1}{y}\right) W^M = \left(W - \frac{1}{y}\right) z.
\end{equation}
Substituting $W = r e^{i\theta}$ yields the parametric expression
\begin{equation}\label{r_para}
	r(\theta) = \frac{\left(\frac{1}{y} - 1\right)\sin((M-1)\theta) + \frac{1}{y}\sin((M+1)\theta) + \sqrt{\Delta}}{2\sin(M\theta)},
\end{equation}
where the discriminant $\Delta$ is given by
\begin{align}
	\Delta &= \left(\frac{1}{y} - 1\right)^2 \sin^2((M-1)\theta) + \frac{1}{y^2} \sin^2((M+1)\theta)+  \frac{1-y}{y^2}\big(\cos(2M\theta)+\cos(2\theta)-2\big).
\end{align}
Furthermore, we derive the parametrization
\begin{equation}\label{x_theta}
	x(\theta) = \frac{y^M}{r \sin\theta} \left( r^{M+1} \sin((M+1)\theta) + \left(1 - \frac{1}{y}\right) r^M \sin(M\theta) \right),
\end{equation}
from which the limiting spectral density follows as
\begin{equation}
	\rho(\theta) = \frac{1}{\pi y^M} \frac{(r \sin\theta)^2}{r^{M+1} \sin((M+1)\theta) + \left(1 - \frac{1}{y}\right) r^M \sin(M\theta)}, \quad \theta \in (0,\pi).
\end{equation}
We next study the support of the limiting spectral distribution. Evaluating at the endpoints
\begin{align}
	r(0) &= \frac{M+1+2M\left(\frac{1}{y}-1\right)+\sqrt{(M+1)^2+4M\left(\frac{1}{y}-1\right)}}{2M}, \\
	r(\pi) &= \frac{-M-1-2M\left(\frac{1}{y}-1\right)+\sqrt{(M+1)^2+4M\left(\frac{1}{y}-1\right)}}{2M}.
\end{align}
The spectral edges are given by
\begin{align}
	x_{\pm} &= \frac{y^{M+1}}{2^{M+1}M^M} \left(M+1+2\left(\frac{1}{y}-1\right)\pm\sqrt{(M+1)^2+4M\left(\frac{1}{y}-1\right)}\right) \notag \\
	&\quad \times \left(M+1+2M\left(\frac{1}{y}-1\right)\pm\sqrt{(M+1)^2+4M\left(\frac{1}{y}-1\right)}\right)^M.
\end{align}
These spectral edges can also be determined from the algebraic resolvent formula \eqref{algformula} (see e.g., \cite{AIK13}). The edges satisfy the bounds
\begin{equation}
	0 \leq x_{-} < 1 < x_{+} \leq \frac{(M+1)^{M+1}}{M^M},
\end{equation}
where equality holds if and only if $y=1$.
The parametrization $x(\theta)$ is  decreasing in $\theta$, establishing a bijection between $(0,\pi)$ and $(x_{-},x_{+})$. Thus, for each $x_0 \in (x_{-},x_{+})$, there exists a unique $\psi \in (0,\pi)$ such that $x_0 = x(\psi)$.

For a determinantal point process with correlation kernel $K_{M,N}(x,y)$, the $n$-point correlation functions are given by
\begin{equation}
	R_{M,N}^{(n)}(x_1,\dots,x_n) = \det\left[K_{M,N}(x_i,x_j)\right]_{i,j=1}^n.
\end{equation}
Recall the definition of the sine kernel (see e.g., \cite{AGZ10})
\begin{equation}
	K_{\text{sin}}(x,y) = \frac{\sin(\pi(x-y))}{\pi (x-y)}.
\end{equation}
We now present our main results for the case $y \neq 1$. For the special case $y = 1$, we refer the reader to \cite{LWW23, LWZ16}.
\begin{theorem}\label{bulkthm}
	Assume $\lim_{N \to \infty} \Delta_{M,N} = 0$. For $\theta \in (0, \pi)$, let
	\begin{equation}\label{subcgd}
		x_i = \sum_{j=1}^M\log(N+v_j) + \log x(\theta) + \frac{\xi_i}{\rho_{M,N}}, \quad i=1,2,\dots,n,
	\end{equation}
 and the scaling factor is
	\begin{equation}
		\rho_{M,N} = \rho_{M,N}(\theta) = \frac{N r \sin\theta}{\pi},
	\end{equation}
		where $x(\theta)$ is defined in \eqref{x_theta} and $r$ is given by \eqref{r_para}. The following limits for correlation functions of eigenvalues of $\log(Y^*_M Y_M)$ 
	\begin{equation}
		\lim_{N \to \infty} (\rho_{M,N})^{-n} R^{(n)}_{M,N}(x_1,\dots,x_n) = \det\left[K_{\mathtt{sin}}(\xi_i,\xi_j)\right]_{i,j=1}^n,
	\end{equation}
hold uniformly for $\xi_1,\dots,\xi_n$ in any compact subset of $\mathbb{R}$.
\end{theorem}

The limiting spectral density can be derived from the limit of the 1-point correlation function.

\begin{corollary}
	For $\theta \in (0, \pi)$, let
	\begin{equation}
		x_1 = \sum_{j=1}^M\log(N+v_j) + \log x(\theta) + \frac{\xi_1}{\rho_{M,N}},
	\end{equation}
	with $\xi_1$ in any compact subset of $\mathbb{R}$. Then the limiting mean density satisfies
	\begin{equation}
		\lim_{N \to \infty} \frac{1}{N x(\theta)} R^{(1)}_{M,N}(x_1) = \lim_{N \to \infty} \frac{1}{N x(\theta)} K_{M,N}(x_1,x_1) = \rho(\theta).
	\end{equation}
\end{corollary}

\begin{remark}
	The proof of Theorem \ref{bulkthm} (or, equivalently, the results in \cite{LWW23}) shows that the conclusion remains valid under the more general conditions $1 \ll M \ll N$ and $\Delta_{M,N} \to 0$.
\end{remark}
These results extend classical results from square matrix products to a special class of rectangular ensembles, with potential implications for stability analysis in deep neural networks and communication systems. Our central contribution is the novel parameterization of the limiting spectral density. This framework enables the contour construction and saddle-point analysis necessary to derive the bulk statistics-a proof of considerable complexity.
	\section{Proof of Theorem    \ref{bulkthm}}\label{sect.2}
The asymptotic analysis for the $n$-correlation functions of $\log(Y_M^*Y_M)$ depends on the kernel \eqref{logkn}, which is our main focus. The primary distinction of our approach from the analogous steepest descent method in \cite{LWW23} rests on two aspects: firstly, the application of our  limiting spectral density for scaling, and secondly, the parameterization-guided construction of the integration contour that validates the saddle-point method.

We begin by outlining the proof strategy. The contour $\Sigma$ is partitioned into ``out" and ``in" parts, denoted by $\Sigma_{\text{out}}$ and $\Sigma_{\text{in}}$, respectively. The kernel decomposes as $K_{M,N}(x,y) = I_1 + I_2$, where:
\begin{itemize}
	\item $I_1$, defined on $\mathcal{C} \times \Sigma_{\text{out}}$, is evaluated via a saddle-point analysis and shown to be negligible;
	\item $I_2$, defined on $\mathcal{C} \times \Sigma_{\text{in}}$, is computed using Cauchy's theorem and is shown to yield the principal contribution .
\end{itemize}
\begin{proof}[Proof of Theorem \ref{bulkthm}]
	We determine the asymptotic behavior of the integral using a three-step method.
	
	\noindent\textbf{Step 1: Contour constructions and integral decomposition.}
After making the substitution
\begin{align}
	x &= g(\xi) = \sum_{j=1}^M \log(N + v_j) + \log x(\theta) + \frac{\xi}{\rho_{M,N}}, \\
	y &= g(\eta) = \sum_{j=1}^M \log(N + v_j) + \log x(\theta) + \frac{\eta}{\rho_{M,N}},
\end{align}
and applying the change of variables $s \to sN$, $t \to tN$,
we obtain
\begin{equation}
	K_{M,N}(g(\xi), g(\eta)) = N \int_{c-i\infty}^{c+i\infty} \frac{ds}{2\pi i} \oint_{\Sigma_{1/N}} \frac{dt}{2\pi i} \frac{1}{s - t} e^{f_{M,N}(s) - f_{M,N}(t)} e^{N \frac{\xi t - \eta s}{\rho_{M,N}}}.
\end{equation}
Here, $\Sigma_{1/N}$ is a positively oriented contour enclosing $0, -1/N, -2/N, \dots, -(N-1)/N$, and $c$ is chosen such that the vertical contour avoids $\Sigma_{1/N}$. The function $f_{M,N}(z)$ is defined by
\begin{equation}
	f_{M,N}(z) = \sum_{j=0}^M \log \Gamma(zN + N + v_j) - \log \Gamma(zN) - zN\left( \log x(\theta) + \sum_{j=1}^M \log(N + v_j) \right).
\end{equation}
Using Stirling's formula, we have the uniform expansion as $z \to \infty$ in the sector $|\arg(z)| \leq \pi - \epsilon$:
\begin{equation}
	\log \Gamma(z) = \left(z - \tfrac{1}{2}\right) \log z - z + \log \sqrt{2\pi} + \frac{1}{12z} + O\left(\frac{1}{z^3}\right).
\end{equation}
This yields the asymptotic expansion
\begin{equation}
	f_{M,N}(z) = N g_M(z; \theta) + c_{M,N}(z) + O(\Delta_{M,N}),
\end{equation}
where
\begin{align}\label{defg_M}
	g_M(z; \theta) &= (z+1)(\log(z+1) - 1) + M\left(z + \tfrac{1}{y}\right)\left(\log\left(z + \tfrac{1}{y}\right) - 1\right) \notag \\
	&\quad - z(\log z - 1) + M z \log y - z \log x(\theta),
\end{align}
and
\begin{equation}
	c_{M,N}(z) = \sum_{j=1}^M N_j \log N - \frac{1}{2} \sum_{j=0}^M \log(z + N_j) + \frac{1}{2} \log z + \frac{M}{2} \log(\sqrt{2\pi}).
\end{equation}
The derivatives of $g_{M}(z;\theta)$ are
\begin{align}
	g'_M(z; \theta) &= \log \left( \frac{(z+1)(z + \frac{1}{y})^M y^M}{z x(\theta)} \right), \\
	g''_M(z; \theta) &= \frac{1}{z+1} + \frac{M}{z + \frac{1}{y}} - \frac{1}{z}.
\end{align}
For $\theta \in (0, \pi)$, we define
\begin{equation}
	q_M(\theta) = r(\theta) e^{i\theta} - \frac{1}{y}, \quad h_M(\theta) = \Re q_M(\theta),
\end{equation}
and note that $h_M$ maps $[0, \pi)$ bijectively to $(r(\pi) - 1/y, r(0) - 1/y]$. From the resolvent equation \eqref{re_equa}, it is easy to see that
\begin{equation*}
	g'_M(q_M(\theta); \theta) = 0.
\end{equation*}
Based on the properties of $g_M(z;\theta)$, we now define the contours
\begin{align*}
	\Sigma_+ &= \{ t = q_M(\theta) \mid \theta \in [0, \pi) \}, \\
	\Sigma_- &= \{ t = \overline{q_M(\theta)} \mid \theta \in [0, \pi) \}.
\end{align*}
Let $C \in (0,1)$ (independent of $M$ and $N$) be such that
$\sum_{j=0}^M \log\left| \frac{-CN + N + v_j}{N + v_j} \right| + \log(x(\theta)) < 0$.
Since $h_M$ is bijective, its inverse $h_M^{-1}$ is well-defined. Then, for a sufficiently small $\delta > 0$, we define the following contours using $\Sigma_\pm$:
\begin{align*}
	\Sigma^1_\pm &= \left\{ t \in \Sigma_\pm \mid 0 \leq \pm \arg t \leq \theta - \delta \right\}, \\
	\Sigma^2_\pm &= \left\{ t \in \Sigma_\pm \mid \pm \arg t > \theta + \delta \text{ and } \Re t \geq -C \right\}, \\
	\Sigma^3_\pm &= \left\{ t = x \pm i \Im q_M(h_M^{-1}(-C)) \mid x \in [-1 + \tfrac{1}{2N}, -C] \right\}, \\
	\Sigma^4 &= \left\{ t = -1 + \tfrac{1}{2N} + iy \mid y \in [-\Im q_M(h_M^{-1}(-C)), \Im q_M(h_M^{-1}(-C))] \right\}.
\end{align*}
Let $\Sigma_L$ be the vertical segment connecting the right endpoints of $\Sigma^2_+$ and $\Sigma^2_-$, and let $\Sigma_R$ be the vertical segment connecting the left endpoints of $\Sigma^1_+$ and $\Sigma^1_-$. Therefore, we can deform $\Sigma_{1/N}$ into
\begin{equation}
	\Sigma_1^+ \cup \Sigma_2^+ \cup \Sigma_3^+ \cup \Sigma_4 \cup \Sigma_1^- \cup \Sigma_2^- \cup \Sigma_3^- \cup \Sigma_L \cup \Sigma_R.
\end{equation}
The contour for $s$ is chosen as the vertical line
\begin{equation}
	C_\theta = \{ h_M(\theta) + iy \mid y \in \mathbb{R} \}.
\end{equation}
Thus, $\Sigma_{1/N}$ is the union of two separate closed contours $\Sigma_1^+ \cup \Sigma_R \cup \Sigma_1^-$ and $\Sigma_2^+ \cup \Sigma_3^+ \cup \Sigma_4 \cup \Sigma_3^- \cup \Sigma_2^- \cup \Sigma_L$. The contour for $s$ lies between these two closed contours; see Figure \ref{fig}. Both contours are positively oriented.
\begin{figure}
	\centering
	\includegraphics[width=0.7\textwidth]{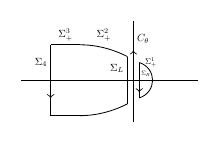}
	\caption{Schematic contours in the proof of Theorem \ref{bulkthm}}
	\label{fig}
\end{figure}
We divide $\Sigma_{1/N}$ into the ``outer'' part $\Sigma_{\text{out}} = \Sigma_1^+ \cup \Sigma_2^+ \cup \Sigma_3^+ \cup \Sigma_4 \cup \Sigma_1^- \cup \Sigma_2^- \cup \Sigma_3^-$ and the ``inner'' part $\Sigma_{\text{in}} = \Sigma_L \cup \Sigma_R$. Accordingly, the kernel $\frac{1}{N}K_{M,N}(x,y)$ splits into two parts
\begin{align}
	I_1 &= \lim_{\delta \to 0} \int_{C_\theta} \frac{ds}{2\pi i} \oint_{\Sigma_{\text{out}}} \frac{dt}{2\pi i} \frac{1}{s - t} e^{f_{M,N}(s) - f_{M,N}(t)} e^{N \frac{\xi t - \eta s}{\rho_{M,N}}}, \\
	I_2 &= \lim_{\delta \to 0} \int_{C_\theta} \frac{ds}{2\pi i} \oint_{\Sigma_{\text{in}}} \frac{dt}{2\pi i} \frac{1}{s - t} e^{f_{M,N}(s) - f_{M,N}(t)} e^{N \frac{\xi t - \eta s}{\rho_{M,N}}}.
\end{align}

\subsection*{Step 2: Asymptotics of the local part}

Applying the residue theorem, we obtain
\begin{align}
	I_2 &= -\int_{q_M(-\theta)}^{q_M(\theta)} \frac{ds}{2\pi i} e^{N \frac{\xi - \eta}{\rho_{M,N}} s} \sin \left( N \frac{\xi - \eta}{\rho_{M,N}} \Im q_M(\theta) \right) \notag \\
	&= -\rho_{M,N} \frac{1}{N} e^{N \frac{\xi - \eta}{\rho_{M,N}} \Re q_M(\theta)} \frac{\sin \left( N \frac{\xi - \eta}{\rho_{M,N}} \Im q_M(\theta) \right)}{\pi (\xi - \eta)} \notag \\
	&= -\rho_{M,N} \frac{1}{N} \exp \left( \pi (\xi - \eta) \left( \cot \theta - \frac{1}{y r(\theta) \sin \theta} \right) \right) \frac{\sin (\pi (\xi - \eta))}{\pi (\xi - \eta)}.
\end{align}
	
	\noindent\textbf{Step 3: Estimates of the global part.}
	We now prove the estimate
	\begin{equation*}
		I_1 = \mathcal{O}(N^{-2/5}).
	\end{equation*}
	The remaining part involves steepest-descent analysis. We need estimates for $\Re f_{M,N}(t)$ on $\Sigma_{\text{out}}$ and $\Re f_{M,N}(s)$ on $C_\theta$.

	\noindent\textbf{Global estimates:}
	
	\begin{enumerate}
		\item For $t \in \Sigma_\pm^3$, we examine the monotonicity of $\Re f_{M,N}(x+iy)$ with respect to $x$. By
		\begin{equation}
			\frac{d}{du} \Re f_{M,N}(u+iv) = \Re \left( f_{M,N}'(u+iv) \right),
		\end{equation}
		and
		\begin{equation}
			\frac{1}{N} f_{M,N}'(z) = \sum_{j=0}^M \psi(zN+N_j) - \log(x(\theta)) - \psi(zN) - \sum_{j=1}^M \log(N_j),
		\end{equation}
		and using the asymptotic expansion of $\psi$ and the assumption on $C \in (0,1)$, we find that $\Re (f_{M,N}'(t)) < 0$ for $t \in \Sigma_\pm^3$. Hence, $\Re f_{M,N}(t)$ is strictly decreasing as a function of $x$ along $\Sigma_\pm^3$, and thus attains its minimum at the right endpoints $-C \pm i \Im q_M(h_M^{-1}(-C))$.
		
		\item For $t \in \Sigma^4$, we examine the monotonicity along the $y$-axis.
		By	\begin{equation}
			\frac{d\Re f_{M,N}(u+iv)}{dv}=-\Im \frac{df_{M,N}(z)}{dz}|_{z=u+iv}
		\end{equation}
	and 	\begin{equation}
		\Im f'_{M,N}(z)=-\sum_{j=0}^M \sum_{n=0}^{\infty}\Im	\frac{1}{tN+N+v_j+n}+\sum_{n=0}^{\infty}\Im \frac{1}{tN+n}	\end{equation}
		 let $\sigma = tN+N+v_j$, so $\Re \sigma = \frac{1}{2}+v_j$. Then,
		For $t \in \Sigma^4$, $\Im f'_{M,N}(t) > 0$ if $\Im t > 0$, and $\Im f'_{M,N}(t) < 0$ if $\Im t < 0$. Hence, $\Re f_{M,N}(t)$ attains its maximum at $t =  -1+ \frac{M+1}{2N}$. Moreover, $\Re f_{M,N}(-1 + \frac{1}{2N}) < \Re f_{M,N}(q_M(\pm \theta)) - \epsilon$ for some $\epsilon > 0$.
		
We require the following lemma to verify the extrema of $g_M$ with respect to $s \in C_\theta$ and $t \in \Sigma_{\pm}$. Its proof will be provided at the end of this section.

\begin{lemma}\label{glolemma}
	Let $g_M(t; \theta)$ be as defined in \eqref{defg_M}. Then:
	\begin{enumerate}
		\item Let $\theta \in (0, \pi)$. As $t$ traverses $\Sigma_+$ (respectively, $\Sigma_-$), the real part $\Re g_M(t; \theta)$ attains its unique minimum at $q_M(\theta)$ (respectively, $\overline{q_M(\theta)}$). Moreover,
		\begin{equation*}
			\frac{d}{d\phi} \Re g_M(q_M(\phi); \theta)
			\begin{cases}
				< 0, & \phi \in (0, \theta), \\
				> 0, & \phi \in (\theta, \pi),
			\end{cases}
			\quad
			\frac{d}{d\phi} \Re g_M(\overline{q_M(\phi)}; \theta)
			\begin{cases}
				< 0, & \phi \in (0, \theta), \\
				> 0, & \phi \in (\theta, \pi).
			\end{cases}
		\end{equation*}
		
		\item Let $\theta \in (0, \pi)$. As $s$ traverses $C_\theta$, the real part $\Re g_M(s; \theta)$ attains its global maximum at $q_M(\theta)$ and $\overline{q_M(\theta)}$. Moreover,
		\begin{equation*}
			\frac{d}{dv} \Re g_M(\Re q_M(\theta) + iv; \theta)
			\begin{cases}
				< 0, & v > \Im q_M(\theta), \\
				> 0, & v \in (0, \Im q_M(\theta)), \\
				< 0, & v \in (\Im \overline{q_M(\theta)}, 0), \\
				> 0, & v < \Im \overline{q_M(\theta)}.
			\end{cases}
		\end{equation*}
	\end{enumerate}
\end{lemma}

	\item For $t \in \Sigma_\pm^1 \cup \Sigma_\pm^2$, the function $f_{M,N}(t)$ is uniformly approximated by $g_M$. By part $(a)$ of Lemma~\ref{glolemma}, $g_M(t;\theta)$ attains its unique minimum on $\Sigma_{+}^1 \cup \Sigma_{+}^2$ at $q_M(\theta)$ and its unique minimum on $\Sigma_{-}^1 \cup \Sigma_{-}^2$ at $\overline{q_M(\theta)}$.
	
	\item For $s \in C_\theta$, we partition $C_\theta$ into two segments $C_\theta^1 = \{ s \in C_\theta \mid |\Im s| \leq K \}, \quad
		C_\theta^2 = C_\theta \setminus C_\theta^1$
	where $K$ is a sufficiently large positive constant such that $K > \Im q_M(h_M^{-1}(-C))$. 
	For $s \in C_\theta^1$, $f_{M,N}(s)$ is approximated by $g_M$, which, by part $(b)$ of Lemma~\ref{glolemma}, attains its maximum at $q_M(\theta)$ and $\overline{q_M(\theta)}$. 
	For $s \in C_\theta^2$, the asymptotic expansion applies. Moreover, for sufficiently large $K$, we have $\Im f'_{M,N}(s) > 0$ when $s$ lies in the upper part of $C_\theta^2$ (i.e., $\Im s > 0$), and $\Im f'_{M,N}(s) < 0$ when $s$ lies in the lower part of $C_\theta^2$ (i.e., $\Im s < 0$). 
	Consequently, $f_{M,N}(s)$ decreases at least linearly as $s \to \pm i\infty$ along $C_\theta^2$.
\end{enumerate}
To estimate $\Re f_{M,N}(t)$ and $\Re f_{M,N}(s)$ locally around $q_M(\theta)$ and $\overline{q_M(\theta)}$, we divide the contour $C_\theta$ into 
\begin{align*}
	C_{\text{local},+} &= C_\theta \cap B(q_M(\theta), N^{-2/5}), \\
	C_{\text{local},-} &= C_\theta \cap B(\overline{q_M(\theta)}, N^{-2/5}), \\
	C_{\text{global}} &= C_\theta \setminus (C_{\text{local},+} \cup C_{\text{local},-}),
\end{align*}
and similarly divide $\Sigma_{\text{out}}$ into
\begin{align*}
	\Sigma_{\text{local},+} &= \Sigma_{\text{out}} \cap B(q_M(\theta), N^{-2/5}), \\
	\Sigma_{\text{local},-} &= \Sigma_{\text{out}} \cap B(\overline{q_M(\theta)}, N^{-2/5}), \\
	\Sigma_{\text{global}} &= \Sigma_{\text{out}} \setminus (\Sigma_{\text{local},+} \cup \Sigma_{\text{local},-}).
\end{align*}
Additionally, we define
\[
\Sigma_{\text{local},+}^0 = \Sigma_+ \cap B(q_M( \theta), N^{-2/5}),\quad \Sigma_{\text{local},-}^0 = \Sigma_- \cap B(\overline{q_M( \theta)}, N^{-2/5})
\]

Around $q_M(\theta)$ and $\overline{q_M(\theta)}$, the function $f_{M,N}(t)$ is approximated by $g_M(t; \theta)$. Since $g'_M(q_M(\theta); \theta) = 0$ and $g'_M(\overline{q_M(\theta)}; \theta) = 0$, within the balls $B(q_M(\theta), N^{-2/5})$ and $B(\overline{q_M(\theta)}, N^{-2/5})$, we have the expansions
\[
f_{M,N}(t) = f_{M,N}(q_M(\theta)) + \frac{N g''_M(q_M(\theta))}{2} (t - q_M(\theta))^2 + \mathcal{O}(N^{-1/5}), 
\]
and
\[
f_{M,N}(t) = f_{M,N}(\overline{q_M(\theta)}) + \frac{N g''_M(\overline{q_M(\theta)})}{2} (t - \overline{q_M(\theta)})^2 + \mathcal{O}(N^{-1/5}).
\]

\begin{lemma}\label{lolemma}
	There exists $\epsilon > 0$ such that for all sufficiently large $M$, the following inequalities hold
	\begin{align*}
		\Re \left( g_M(s) - g_M(q_M(\theta)) \right) &\leq -\epsilon |s - q_M(\theta)|^2, \quad \text{for } s \in C_{\text{local},+}, \\
		\Re \left( g_M(t) - g_M(q_M(\theta)) \right) &\geq \epsilon |t - q_M(\theta)|^2, \quad \text{for } t \in \Sigma_{\text{local},+}^0,
	\end{align*}
	and similarly for $\overline{q_M(\theta)}$.
\end{lemma}

We now estimate the integral over $C_{\text{local},+} \times \Sigma_{\text{local},+}$; the case for $\overline{q_M(\theta)}$ is analogous. By Lemma \ref{lolemma}, we have
\begin{align}
	&e^{-N \frac{\xi - \eta}{\rho_{M,N}} q_M(\theta)} \lim_{\delta \to 0} \int_{C_{\text{local},+}} \frac{ds}{2\pi i} \int_{\Sigma_{\text{local},+}} \frac{dt}{2\pi i} \frac{1}{s - t} e^{f_{M,N}(s) - f_{M,N}(t)} e^{N \frac{\xi t - \eta s}{\rho_{M,N}}} \notag\\
	&= \mathrm{P.V.} \int_{\Sigma_{\text{local},+}} \frac{dt}{2\pi i} \int_{C_{\text{local},+}} \frac{ds}{2\pi i} \frac{1}{s - t} e^{f_{M,N}(s) - f_{M,N}(t)} e^{N \frac{\xi(t - q_M(\theta)) - \eta(s - q_M(\theta))}{\rho_{M,N}}} \notag\\
	&= \mathrm{P.V.} \int_{\Sigma^0_{\text{local},+}} \frac{dt}{2\pi i} \int_{C_{\text{local},+}} \frac{ds}{2\pi i} \frac{1}{s - t} e^{ \frac{N g''_M(q_M(\theta))}{2} (s - q_M(\theta))^2 - N \frac{\eta(s - q_M(\theta))}{\rho_{M,N}} } \notag\\
	&\quad \times e^{ -\frac{N g''_M(q_M(\theta))}{2} (t - q_M(\theta))^2 + N \frac{\xi(t - q_M(\theta))}{\rho_{M,N}} } \left( 1 + \mathcal{O}(N^{-1/5}) \right),
\end{align}
where the error term is uniform.
Let $w = (s - q_M(\theta)) \sqrt{N}$ and $z = (t - q_M(\theta)) \sqrt{N}$. Then we obtain
\begin{align}
	&\int_{C_{\text{local},+}} \frac{ds}{2\pi i} \frac{1}{s - t} e^{ \frac{N g''_M(q_M(\theta))}{2} (s - q_M(\theta))^2 - N \frac{\eta(s - q_M(\theta))}{\rho_{M,N}} } \left( 1 + \mathcal{O}(N^{-1/5}) \right) \notag \\
	&= \int_{-N^{1/10}i}^{N^{1/10}i} \frac{dw}{2\pi i} \frac{1}{w - z} e^{ \frac{g''_M(q_M(\theta))}{2} w^2 - \sqrt{N} \frac{\eta w}{\rho_{M,N}} } \left( 1 + \mathcal{O}(N^{-1/5}) \right) = \mathcal{O}(1),
\end{align}
uniformly for $t \in \Sigma^0_{\text{local},+} \setminus \{q_M(\theta)\}$. The principal value integral then becomes
\begin{equation}
	\mathrm{P.V.} \int_{\Sigma_{\text{local},+}} \frac{dt}{2\pi i} \mathcal{O}(1) e^{ -\frac{N g''_M(q_M(\theta))}{2} (t - q_M(\theta))^2 + N \frac{\xi(t - q_M(\theta))}{\rho_{M,N}} } = \mathcal{O}(N^{-2/5}).
\end{equation}
Thus, the double contour integral over $C_{\text{local},\pm} \times \Sigma^0_{\text{local},\pm}$ is $\mathcal{O}(N^{-2/5})$. Combining Lemmas \ref{glolemma} and \ref{lolemma}, there exists $\epsilon > 0$ such that for $t \in \Sigma_{\text{global}}$ and $s \in C_{\text{global}}$
\begin{equation}
	\Re f_{M,N}(s) + \epsilon N^{1/5} < \Re f_{M,N}(q_M(\theta)) \quad \text{and} \quad \Re f_{M,N}(\overline{q_M(\theta)}) < \Re f_{M,N}(t) - \epsilon N^{1/5},
\end{equation}
and $\Re f_{M,N}(s) \to -\infty$ rapidly as $\Im s \to \pm \infty$ along $C_{\text{global}}$. Hence,
\begin{equation}
	\left| \int_{C_\theta \times \Sigma_{\text{out}} \setminus (C_{\text{local},+} \times \Sigma_{\text{local},+} \cup C_{\text{local},-} \times \Sigma_{\text{local},-})} \frac{ds}{2\pi i} \frac{dt}{2\pi i} \frac{1}{s - t} e^{f_{M,N}(s) - f_{M,N}(t)} e^{N \frac{\xi t - \eta s}{\rho_{M,N}}} \right| = \mathcal{O}(e^{-\epsilon N^{1/5}}),
\end{equation}
since $|s - t|^{-1} = \mathcal{O}(N^{-2/5})$. This proves that $I_1 = \mathcal{O}(N^{-2/5})$, thus completing the proof.
	
\end{proof}
\begin{proof}[Proof of Lemma \ref{glolemma}]

Recall that $x(\theta)$ is defined as in \eqref{x_theta}. A direct computation shows that
\begin{equation}
	\frac{d}{d\phi} g_M(q_M(\phi); \theta) = (x(\phi)-x(\theta)) \frac{d}{d\phi} q_M(\phi). 
	\tag{2.152}
\end{equation}
One can show that both $x(\phi)$ and $\Re q_M(\phi)$ are decreasing functions for $\phi \in [0, \pi)$. Then, part (a) follows.

To prove part (b), we require the following two computations
\begin{align}
	\frac{d}{dv} \Re g_M(u + iv; \theta) &= -\Im \left. \frac{d}{dz} g_M(z; \theta) \right|_{z=u+iv}\notag \\  
	&= -\arctan\frac{v}{u+1} - M\arctan\frac{v}{u+\frac{1}{y}} + \arctan\frac{v}{u},
\end{align}
and
\begin{align}\label{g2}
	\frac{d^2}{dv^2} \Re g_M(u + iv; \theta) &= -\Re \left. \frac{d^2}{dz^2} g_M(z; \theta) \right|_{z=u+iv} \notag \\
	&= -\frac{u+1}{(u+1)^2-v^2} - M\frac{u+\frac{1}{y}}{(u+\frac{1}{y})^2-v^2} + \frac{u}{u^2-v^2}.
\end{align}
For $\theta \in (0, \pi)$, we have $\Re q_M(\theta) < r(0) - \frac{1}{y}$. From \eqref{g2}, there exists a unique $v_0 > 0$ such that
\begin{equation}
	\left. \frac{d^2}{dv^2} \Re g_M(\Re q_M(\theta) + iv; \theta) \right|_{v=v_0} = 0. 
\end{equation}
It is straightforward to see that $\frac{d}{dv} \Re g_M(\Re q_M(\theta) + iv; \theta)$ is increasing for $v \in (0, v_0)$ and decreasing for $v > v_0$. Combining this with the fact that
\begin{equation}
	\left. \frac{d}{dv} \Re g_M(u + iv; \theta) \right|_{v=0} = \left. \frac{d}{dv} \Re g_M(u + iv; \theta) \right|_{v=\Im q_M(\theta)} = 0, 
\end{equation}
and analyzing the signs of the first and second derivatives of $\Re g_M(z;\theta)$, we deduce the monotonicity of the function in the corresponding intervals. Thus, the first two inequalities in part (b) follow. By the symmetry of $\Re g_M(u + iv; \theta)$, the remaining two inequalities also hold.
\end{proof}
	\begin{proof}[Proof of Lemma \ref{lolemma}]
We now prove the case for $q_M(\theta)$; the case for $\overline{q_M(\theta)}$ follows similarly. From the expression
\begin{equation}
	g''_M(q_M(\theta);\theta) = \frac{1}{q_M(\theta)+1} + \frac{M}{q_M(\theta)+\frac{1}{y}} - \frac{1}{q_M(\theta)},
\end{equation}
we observe that for any fixed and sufficiently large $M$, $\Re g''_M(q_M(\theta);\theta) > 0$. Consequently, along $C_{\text{local},+}$, the real part $\Re g_M(s;\theta)$ decreases quadratically as we move away from $q_M(\theta)$.

On the other hand, the curve $\Sigma_{\text{local},+}^0$ has tangent direction at $q_M(\theta)$ given by
\begin{equation}
	\arg q'_M(\theta) = \arg(\cos\phi + i\sin\phi).
\end{equation}
The second-order directional derivative of $\Re g_M(z;\theta)$ at $q_M(\theta)$ in the direction $(\cos\phi, \sin\phi)$ is computed as
\begin{align}
	D^2_{\phi} \Re g_M(q_M(\theta); \theta) &= \Re g''_M(q_M(\theta);\theta) (\cos^2\phi - \sin^2\phi) - 2\Im g''_M(q_M(\theta);
	\theta) \cos\phi \sin\phi \notag \\
	&= \Re \left( g''_M(q_M(\theta);\theta) (\cos\phi + i\sin\phi)^2 \right).
\end{align}
A direct computation shows that $D^2_{\phi} \Re g_M(q_M(\theta); \theta) > 0$. Therefore, along $\Sigma_{\text{local},+}^0$, the real part $\Re g_M(t;\theta)$ increases quadratically as we move away from $q_M(\theta)$.
\end{proof}

\section{Concluding remarks}	
Our results generalize several known results in random matrix theory. For a model formed by multiplying a rectangular matrix with several square matrices, we obtain the limiting spectral density and its parametric representation for $Y_MY_M^{*}$ in the case where all rectangularity parameters are equal ($y_l=y$). The analysis, however, is restricted to this homogeneous case, as determining the global spectral density for models with distinct parameters  $y_l$ remains an open challenge.

 Therefore, future research will be directed toward weakening the assumption of a common asymptotic limit for the rectangularity parameters. This will enable the study of global spectral density and local statistical universality in these broader classes of rectangular random matrix models.
	\begin{acknow}
	 This work was supported by the National Natural Science Foundation of China \#12371157.
	\end{acknow}

\end{document}